\documentclass[11pt, oneside]{article}   	
\usepackage{geometry}                		
\usepackage{graphicx}				
\usepackage{amssymb}
\usepackage{amsmath}
\usepackage{amsthm}
\usepackage{harpoon}
\usepackage{accents}

\usepackage{tikz}  
\usepackage{float}
\restylefloat{table}
\usepackage{mathrsfs}
\usepackage{verbatim}
\usepackage{enumitem}  

\usetikzlibrary{positioning,chains,fit,shapes,calc}  

\newtheorem{theorem}{Theorem}
\newtheorem{lemma}{Lemma}
\newtheorem{definition}{Definition}

\newtheorem{corollary}{Corollary}

\newcommand{\ppp}{\[\begin{aligned}}
\newcommand{\ooo}{\end{aligned}\]}
\newcommand{\err}{\mathrm{err}}

\newcommand{\mrm}[1]{\mathrm{#1}}

\newcommand{\ti}{\tilde{h}}
\newcommand{\mult}{\mathrm{mult}}

\newcommand{\semi}{\mathrm{semi}}
\newcommand{\tCH}{\tilde{\mathrm{CH}}_2}
\newcommand{\CH}{\mathrm{CH}_2}
\newcommand{\te}{\tilde{e}}

\begin{document}

\title{On the leading and penultimate leading  coefficients for NRS(2) applied to a cubic polynomial}
\author{Mario DeFranco}

\maketitle

\abstract{We prove that the leading and penultimate leading coefficients in $u_3$ of the ``error" terms of NRS(2) applied to a cubic polynomial $f(z) =\sum_{i=0}^3 a_i z^i=\prod_{i=1}^3 (1-u_iz)$ with starting point $(-\frac{a_1}{a_2}, -\frac{a_1}{a_2})$ are positive-coefficient polynomials in $u_1$ and $u_2$. Our proof for the leading coefficients simplifies that of \cite{DeFranco} and extends to the penultimate leading coefficients as well.}

\section{Introduction}

We provide a simpler proof of the main result of \cite{DeFranco} and extend it to include the penultimate leading coefficient. We use much of the notation of \cite{DeFranco} and repeat the setup next. 

Let $f(z) \in \mathbb{C}[z]$ be a degree 3 polynomial 
\[
f(z) = \prod_{i=1}^3 (1-u_iz)
\]
We express the difference between the $n$-th iteration and the auxiliary-function zero
\[
(\alpha_0, \alpha_1)
\]
(see \cite{DeFranco 1} and \cite{DeFranco 2}) using polynomials in $u_1, u_2$ and $u_3$. In this paper, we focus on the leading and penultimate leading coefficients of $u_3$ in these polynomials, which are polynomials in $u_1$ and $u_2$. As in \cite{DeFranco}, we use the ring $\tCH$. In section \ref{s ch identities}, we prove certain identities in this ring. In section \ref{s ms}, we use multisets to express certain elements of $\tCH$ that have non-negative coefficients. Our main results are in section \ref{s formulas}, which simplify the recurrence relations and prove that they have positive coefficients.. 

\section{Relations in $\tilde{\mathrm{CH}}_2$} \label{s ch identities}

\begin{lemma} \label{l hab}
\begin{align}
\ti_A \ti_B - \ti_{A-1} \ti_{B-1} = \ti_{B+A} \label{11}\\ 
\ti_A \ti_B - \ti_{A-1} \ti_{B+1} =\ti_{B-A} \label{1m1}
\end{align}
\end{lemma} 

\begin{lemma} \label{l haba} 
\begin{align}
\tilde{h}_B\tilde{h}_{A_1}  \tilde{h}_{A_2} - \tilde{h}_B\tilde{h}_{A_1-1}  \tilde{h}_{A_2-1} &=  \tilde{h}_B\tilde{h}_{A_1+A_2} \label{hbaa}\\
\tilde{h}_{A_1} \tilde{h}_B \tilde{h}_{A_2-1}+ \tilde{h}_{A_1-1} \tilde{h}_B \tilde{h}_{A_2} - \tilde{h}_{A_1-1}\tilde{h}_1 \tilde{h}_B\tilde{h}_{A_2-1} &= \tilde{h}_B\tilde{h}_{A_1+A_2-1} \label{haba}
\end{align}
\end{lemma} 

\begin{proof}
Equation \eqref{hbaa} follows from Lemma \ref{l hab}, equation \eqref{11}. 
Using 
\[
h_1 h_B = h_{B-1}+h_{B+1}
\] 
we write equation \eqref{haba} as 
\[
(\tilde{h}_{A_1} \tilde{h}_B - \tilde{h}_{A_1-1} \tilde{h}_{B+1})\tilde{h}_{A_2-1} + \tilde{h}_{A_1-1} (\tilde{h}_{B} \tilde{h}_{A_2} - \tilde{h}_{B-1} \tilde{h}_{A_2-1}) 
\]
which by Lemma \eqref{l hab} is equal to 
\begin{equation} \label{h 2 t}
\tilde{h}_{B-A_1}\tilde{h}_{A_2-1} + \tilde{h}_{A_1-1} \tilde{h}_{B+A_2}.  
\end{equation}
Now we consider all the case arising from whether $B-A_1, A_1-1$ and $B$ are less than, equal to, or greater than $-1$. For example, consider the case 
\begin{align*} 
B-A_1 &< -1 \\ 
A_1 -1 &> -1 \\ 
B&< -1.
\end{align*}
Then the left side of equation \eqref{h 2 t} is 
\begin{align*}
\sum_{j=0}^{A_1-B-2} \ti_{A_2-1+B-A_1+2+2j} + \sum_{j=0}^{A_1-1} \ti_{B+A_2-A_1+1+2j} &= -\sum_{j=0}^{-B-2} \ti_{A_1+A_2+B+1+2j} \\ 
&= \ti_B \ti_{A_1+A_2-1}.
\end{align*}
The other cases are similar. This completes the proof. 
\end{proof}

\begin{definition} 
We define a left action of $\mathrm{CH}_2$ on $\tilde{\mathrm{CH}}_2$ for $i\geq 0$ and $g \in \tilde{\mathrm{CH}}_2$, by 
\[
h_i g = \ti_i g.
\]
\end{definition} 

\begin{lemma} \label{l action}
For $g_1,g_2 \in \tCH$, 
\[
g_1 g_2 = L(g_1) g_2.
\]
\end{lemma}
\begin{proof}
It is sufficient to check when $g_1$ and $g_2$ are generators. By definition 
\[
L(\ti_i) = \begin{cases} 0 \text{ if } i=-1 \\ 
h_i \text{ if } i > -1 \\ 
-h_{-i-2} \text{ if } i < -1
\end{cases}.
\]
This is the definition of multiplication in $\tCH$. This completes the proof.  
\end{proof}

\begin{definition} 
For  $g_1,g_2,g_3 \in \tCH$, define 
\begin{align*}
W_0(g_1,g_2,g_3) &= g_2(g_1g_3 - S_{-1}(g_1)S_{-1}(g_3))\\ 
W_1(g_1,g_2,g_3) &= S_{-1}(g_1)g_2g_3+ g_1g_2S_{-1}(g_3) - S_{-1}(g_1)g_2S_{-1}(g_3)
\end{align*}
\end{definition} 

\begin{theorem} \label{t W}
For  $g_1,g_2,g_3 \in \tCH$,
\[
W_1(g_1,g_2,g_3) = S_{-1}(W_0(g_1,g_2,g_3)).
\]
\end{theorem} 
\begin{proof}
We write 
\[
g_i = \sum_{j \in \mathbb{Z}} c_{i,j} \ti_j.
\] 
Then 
\[
 W_0(g_1,g_2,g_3) = \sum_{(A,B,C) \in \mathbb{Z}^3} c_{1,A}c_{2,B}c_{3,C} \ti_B(\ti_A \ti_C  - \ti_{A-1}\ti_{C-1}).
\]
By Lemma \ref{l hab}, equation \eqref{11}, this becomes 
\begin{equation}\label{0 ABC sum}
\sum_{(A,B,C) \in \mathbb{Z}^3} c_{1,A}c_{2,B}c_{3,C} \ti_B \ti_{A+C}.
\end{equation}
Likewise 
 \[
 W_0(g_1,g_2,g_3)  = \sum_{(A,B,C) \in \mathbb{Z}^3} c_{1,A}c_{2,B}c_{3,C}  \ti_{A-1} \ti_B \ti_{C}  +\ti_{A} \ti_B \ti_{C-1} - \ti_{A-1} \ti_1 \ti_B \ti_{C-1}. 
 \]
By Lemma \ref{l hab}, equation \eqref{1m1}, this becomes 
\begin{align} \label{1 ABC sum}
\sum_{(A,B,C) \in \mathbb{Z}^3} c_{1,A}c_{2,B}c_{3,C}\ti_B \ti_{A+C-1} 
\end{align} 
which is equal to $S_{-1}$ applied to the sum \eqref{0 ABC sum}. This completes the proof. 
\end{proof}

\section{Multisets} \label{s ms}
\begin{definition} 
Let $\mathrm{MultiSets}$ denote the set of finite multisets of integers, and $\mathrm{MultiSets}^+ \subset \mathrm{MultiSets}$ denote those multisets consisting of non-negative integers. 
For $M \in \mathrm{MultiSets}$, let $\mult(i,M)$ denote the multiplicity of $i$ in $M$.
For two multisets $M_1$ and $M_2$ in $\mathrm{MultiSets}$,, let $M_1+M_2$ denote the multiset sum
\[
M_1+M_2 = \{ i_1+i_2 \colon i_1 \in M_1, i_2 \in M_2\},
\]
so 
\[
\mult(i,M_1+M_2) = \sum_{k \in \mathbb{Z}} \mult(i-k, M_1) \mult(k, M_2).
\]
For integer $i$, we let the operators $S_i$ act on $\mathrm{MultiSets}$ by 
\[
S_i(M) = M + \{ i\}.
\]

Let $M_1 \cup M_2$ denote multiset union, so that 
\[
\mult(i, M_1 \cup M_2) = \mult(i,M_1)+\mult(i,M_2). 
\]

For integers $a,b$ with $a \equiv b \mod 2$, let $[a,b] \in \mathrm{MultiSets}$
\[
[a,b] = \{ a, a+2, \dots, b-2,b \}.
\]

\end{definition} 

\begin{lemma} \label{l mn}
For integers $b_i \geq a_i, i=1,2$, we have
\begin{align}
[a_1,b_1]+[a_2,b_2] &= \bigcup_{j=0}^{\min(b_1-a_1, b_2-a_2)/2} [a_1+a_2+2j, b_1+b_2-2j]. \label{add intervals}\\ 
\end{align}
\end{lemma}

\begin{proof}
Without loss of generality, assume $b_1-a_1\leq b_2-a_2$. Equation \eqref{add intervals} follows plotting the points 
 \[
 (a_1+2j, a_2+2i)
 \]
 for $0 \leq j \leq (b_1-a_1)/2$ and $0 \leq i \leq (b_2-a_2)/2$ in the $x-y$ plane. Start with a point $(a_1+2j,a_2)$, proceed upwards until you reach the point $(a_1+2j,b_2-2j)$, and then proceed to the right. The coordinates of a point in this path are the addends that sum to a number in $[a_1+a_2+2j, b_1+b_2-2j]$. This is the same reasoning from Lemma 3, iii) of \cite{DeFranco}. This completes the proof. 
\end{proof}


\begin{definition}
For $M \in \mathrm{MultiSets}$, define $\ti(M) \in \tCH$ by 
\[
\ti(M) = \sum_{i \in \mathbb{Z}} \mult(i,M)\ti_i.
\]

Let $\tCH^+ \subset \tCH$ and $\mathrm{CH}_2^+ \subset \mathrm{CH}_2$ denote the sets 
\[
\{ \ti(M) \colon M \in \mathrm{MultiSets}^+ \} \text{ and } \{ h(M) \colon M \in \mathrm{MultiSets}^+ \}
\]
respectively. 
\end{definition}

\begin{definition} \label{d Rc}
For an integer $c$, let $R(c)$ denote the set of multisets $M$ such that 
\[
 \mult(c-i-1,M)  \leq  \mult(c-i,M) \leq \mult(c+i,M)
\]
for all $i \geq 0$. That is, $M \in R(c)$ is of the form 
\begin{equation} \label{Rc}
M = \bigcup_{i=1}^{k_1} \{ m_i\} \cup\bigcup_{i=1}^{k_2} [c-n_i, c+n_i]
\end{equation}
for some integers $m_i \geq c$, $n_i \geq 1$ and $k_1, k_2 \geq 0$. 

Let $\tCH(c) \subset \tCH$  denote the set
\[
\{ \ti(M) \colon M \in R(c) \}.
\]
\end{definition} 

\begin{lemma} \label{l Rc add}
\begin{equation} \label{Rc sub}
R(c+1) \subset R(c)
\end{equation}

For $M_1 \in R(c_1)$ and $M_2 \in R(c_2)$, 
\begin{equation} \label{Rc add}
M_1+M_2 \in R(c_1+c_2).
\end{equation}
\end{lemma}
\begin{proof} 
From formula \eqref{Rc}, if $m_i \geq c+1$ then $m_i \geq c$, and 
\[
[c+1-n_1,c+1+n_i]  = [c-(n_i-1), c+(n_i-1)] \cup \{ c+n_i\} \in R(c).
\]
This proves statement \eqref{Rc sub}. 
From Lemma \ref{l mn}
\begin{align*}
[c_1-n_1, c_1+n_1] + [c_2-n_2,c_2+n_2] &\in \bigcup_{j=1}^{\min(n_1,n_2)} [c_1+c_2 - n_1-n_2+2j, c_1+c_2+n_1+n_2 - 2j] \\ 
 &\in R(c_1+c_2).
\end{align*}
And 
\begin{align*}
&\{c_1+k\} + [c_2-n, c_2+n] \\ 
&= \begin{cases}[c+1+c_2+k-n, c+1+c_2-k+n] \cup [ c+1+c_2-k+n+2,c_1+c_2+k+n] &\text{ if } n \geq k+1 \\  
[ c_1+c_2+k-n,c_1+c_2+k+n] &\text{ if } n \leq k
\end{cases}
\end{align*}
which is also in $R(c_1+c_2)$. The last case is 
\[
\{c_1+k_1\} + \{ c_2+k_2\} =  \{c_1 +c_2+k_1+k_2\}  \in R(c_1+c_2). 
\]
This proves statement \eqref{Rc add}.
\end{proof}

\begin{lemma} \label{l left Rc}
For $M \in R(c)$, $i \geq 0$, we have
\[
\ti_i \ti(M) = \ti([-i,i]+M) \in \tCH(c). 
\]
\end{lemma}
\begin{proof}
Left multiplication by $\ti_i$ is equivalent to applying the operator 
\[
\sum_{j=-i}^i S_{-i+2j}.
\] 
Since $[-i,i]\in R(0)$, by Lemma \ref{l Rc add} we have 
\[
\ti([-i,i]+M) \in R(0+c). 
\]
This completes the proof. 
\end{proof}

\begin{lemma} \label{l pos}
Suppose $g = \ti(M)$ with $M\in R(c)$, $c \geq 0$. Then
\[
L(g) \in \CH^+.
\]
\end{lemma}
\begin{proof}

We must prove for $i \geq 0$ that 
\[
\mult(i,M) - \mult(-i-2,M) \geq 0.
\]
If $i \leq c$, then by definition of $R(c)$
\[
\mult(-i-2,M)  \leq \mult(i,M).
\]
If $i > c$, then with $i = c+k$, 
\begin{align*}
 \mult(i,M) &= \mult(c+k,M) \\ 
 &\geq \mult(c-k,M) \\
 &\geq \mult(-c-k-2,M)\\ 
  &=  \mult(-i-2,M)
 \end{align*}
 where we again used the definition of $R(c)$. This completes the proof. 
\end{proof}


\section{Formulas for leading and penultimate leading coefficients} \label{s formulas}

\begin{definition}

sing the degrees from  Lemma 1 of \cite{DeFranco} we denote the leading and penultimate leading coefficients by 
\begin{align*}
\err'(n,0) &= \sum_{i=0}^{2(3^n-2^n)} e(n,0,i)u_3^{2(3^n-2^n)-i} \\ 
\err'(n,1) &= \sum_{i=0}^{2(3^n-2^n)} e(n,1,i)u_3^{2(3^n-2^n)-i}\\ 
\err'(n,-1) &= \sum_{i=0}^{2(3^n)-1} e(n,-1,i)u_3^{2(3^n)-1-i}.
\end{align*}
We substitute these expressions into the modified recurrent relations of \cite{DeFranco} (equations (5)) and obtain relations for the leading coefficients in Definition \ref{d lc0 rec} and the penultimate leading coefficients in Definition \ref{d lc1 rec}.
\end{definition} 

\subsection{Formula for leading coefficients} 

We take the recurrence relations from Definition 8 of \cite{DeFranco}, but in each term we have changed the order of the factors. Thus the $\tilde{e}(n,i,0)$ defined below  (for $i \in \{ 0,1\}$) differ from the $\tilde{\err}_0(n,i)$ in \cite{DeFranco} because $\tilde{\mathrm{CH}}_2$ is non-commutative, but they have the same image under $L$. 
\begin{definition} \label{d lc0 rec}
For integer $n \geq 0$, we recursively define elements $\te(n,0,0)$ and $\te(n,1,0) \in \tilde{\mrm{CH}}_2$ by 
\begin{align*}
\te(0,0,0) &= \tilde{h}_2\\ 
\te(0,1,0) &= \tilde{h}_1,
\end{align*}
and
\begin{align}
\te(n+1,0,0) &= \te(n+1,0,0) (\te(n+1,0,0)^2 - \te(n+1,1,0)^2)   \label{0 rec}\\ 
\te(n+1,1,0) &= \te(n+1,1,0)\te(n+1,0,0)\te(n+1,0,0)  \nonumber \\ 
&\,\,\,\,\,+\te(n+1,0,0)\te(n+1,0,0)\te(n+1,1,0) \\
&\,\,\,\,\,-\te(n+1,1,0) \tilde{h}_1\te(n+1,0,0)\te(n+1,1,0)\label{1 rec}\\ 
\te(n+1,-1,0) &= \te(n+1,1,0)\te(n+1,0,0)\te(n+1,0,0)  \nonumber \\ 
&\,\,\,\,\,+\te(n+1,0,0)\te(n+1,0,0)\te(n+1,1,0) \\
&\,\,\,\,\,-\te(n+1,1,0) \tilde{h}_1\te(n+1,0,0)\te(n+1,1,0)\\ 
&\,\,\,\,\,+ \te(n,-1,0)^2(\te(n,-1,0)-\te(n,1,0)).
\end{align}

\end{definition}

\begin{theorem} \label{t 0 main}
For $n \geq 0$
\begin{equation} \label{s0}
\te(n,1,0) = S_{-1}(\te(n,0,0)).
\end{equation}
\begin{equation} \label{-1eq1}
\te(n,-1,0) = \te(n,1,0)
\end{equation}
and
\begin{equation} \label{multiset 0}
\te(n,0,0)   = \ti(M_{n,0})
\end{equation}
where $M_{n,0} \in R(2^{n+1})$. 
Furthermore
\begin{equation} \label{e0 form}
\te(n+1,0,0) = L(\te(n,0,0)) \ti(M_{n,0} + M_{n,0}).
\end{equation}

\end{theorem}
\begin{proof}
 For $n=0$, we have 
 \[
 \te(n,-1,0) = \te(n,1,0)  = \ti_1 = S_{-1}(\ti_2)= S_{-1}( \te(n,0,0)) 
 \]
 and $h_2 \in \tCH(2)$.

Assume the statements are true for some $n \geq 0$. Then substituting in the induction hypothesis \eqref{-1eq1} into the recurrence relation for of $\te(n+1,-1,0)$ yields 
\[
\te(n+1,-1,0) = \te(n+1,1,0). 
\]

Then using the induction hypothesis \eqref{s0} we have
\begin{align*}
\te(n+1,0,0) &=  W_0(g_1,g_2,g_3)\\ 
\te(n+1,1,0) &= W_1(g_1,g_2,g_3)
\end{align*}
where 
\[
g_1 = g_2=g_3 =  \te(n,0,0).
\]
Applying Theorem \ref{t W} then completes the induction step for equation \eqref{s0}. 

By the induction hypothesis
\[
\te(n,0,0) = \ti(M_{n,0})
\]
for some $M_{n,0} \in R(2^{n+1})$.
From the proof of Theorem \ref{t W}, 
 \begin{equation} \label{sum AB}
\te(n+1,0,0) = \ti(M_{n,0}) \sum_{(A,B) \in M_{n,0}^2} \ti_{A+B}.
 \end{equation}
By definition 
 \[
 \sum_{(A,B) \in M_{n,0}^2} \ti_{A+B}=\ti(M_{n,0}+M_{n,0}),
 \]
which is in $\tCH(2^{n+2})$ by Lemma \ref{l Rc add}. By Lemma \ref{l left Rc} the product \eqref{sum AB} is also in $\tCH(2^{n+2})$. Left multiplication by $\ti(M_{n,0})$ is the same as left multiplication by $L(\ti(M_{n,0}) )$ by Lemma \ref{l action}. This completes the induction step for statements \eqref{multiset 0} and \eqref{e0 form}.
This completes the proof. 
\end{proof}
\begin{corollary} 
For $i\in \{-1,0,1\}$
\begin{equation}
L(\te(n,i,0)) \in \CH^+.
\end{equation}
By Lemma \ref{l pos}, it is sufficient to show that these elements are in $\tCH(c)$ for some $c \geq 0$. $\te(n,0,0) \in \tCH(c)$ with $c = 2^{n+1}\geq 0$, and $\te(n,1,0) = S_{-1}(\te(n,0,0)) \in \tCH(c)$ with $c = 2^{n+1}-1\geq 0$. And $\te(n,-1,0) = \te(n,1,0)$. This completes the proof. 

\end{corollary}

\subsection{Formula for penultimate leading coefficients}

From Theorem \ref{t 0 main} we have 
\[
\te(n,1,0) = \te(n,-1,0) = S_{-1}( \te(n,0,0)) 
\]
which have been used in expressing the relations below.

\begin{definition} \label{d lc1 rec}

\begin{align*}
\te(0,0,1) &= 0\\ 
\te(0,1,1) &= 0 \\ 
\te(0,-1,1) &= \ti_0
\end{align*}

\begin{align}
\te(n+1,0,1) &= W_0(\te(n,0,0),\te(n,0,0), \te(n,0,1) + S_{-1}(\te(n,0,0)) ) \label{r0 1} \\ 
&\,\,\,\,\, + W_0(\te(n,0,0),\te(n,0,1), \te(n,0,0) ) \label{r0 2} \\
&\,\,\,\,\,+ W_0(\te(n,0,1),\te(n,0,0), \te(n,0,0) )  \label{r0 3} \\ 
\te(n+1,1,1) &= W_1(\te(n,0,0),\te(n,0,0), \te(n,0,1) + S_{-1}(\te(n,0,0)) ) \label{r1 1} \\ 
&\,\,\,\,\, + W_1(\te(n,0,0),\te(n,0,1), \te(n,0,0) ) \label{r1 2} \\
&\,\,\,\,\,+ W_1(\te(n,0,1),\te(n,0,0), \te(n,0,0) )  \label{r1 3} \\ 
\te(n+1,-1,1) &= W_1(\te(n,0,0),\te(n,0,0), \te(n,0,1) + S_{-1}(\te(n,0,0)) ) \label{d 1} \\ 
&\,\,\,\,\, + W_1(\te(n,0,0),\te(n,0,1), \te(n,0,0) ) \label{d 2} \\
&\,\,\,\,\, \te(n,-1,1)\te(n,0,0)^2+\te(n,0,1)\te(n,0,0)S_{-1}(\te(n,0,0)) \label{d 3} \\ 
&\,\,\,\,\, -\te(n,-1,1) h_1 \te(n,0,0)S_{-1}(\te(n,0,0)) \label{d 4} \\   
&\,\,\,\,\, +2S_{-1}(\te(n,0,0)) \te(n,0,0)S_{-1}(\te(n,0,0))  \label{d 5} \\ 
&\,\,\,\,\,-S_{-1}(\te(n,0,0)) \ti_1 S_{-1}(\te(n,0,0)) S_{-1}(\te(n,0,0))  \label{d 6} \\ 
&\,\,\,\,\, S_{-1}(\te(n,0,0)) (\te(n,-1,1)-\te(n,1,1)  )S_{-1}(\te(n,0,0))  \label{d 7} 
\end{align}
\end{definition}

 \begin{theorem} \label{t lc1 s}
\begin{align}
\te(n,1,1) &= S_{-1}(\te(n,0,1) \label{lc1 1}\\
\te(n,-1,1) &= S_{-1}(\te(n,0,1))+ S_{-2}(\te(n,0,0)) \label{lc1 -1}
\end{align}
 \end{theorem}
\begin{proof} 
Both statements are true for $n=0$. Apply Theorem \ref{t W} to each of the pairs \eqref{r0 1} and \eqref{r1 1},  \eqref{r0 2} and \eqref{r1 2}, and  \eqref{r0 3} and \eqref{r1 3}. This proves statement \eqref{lc1 1}.  

Now to prove statement \eqref{lc1 -1}, we use induction on $n$. Assume the statement is true for some $n \geq 0$. By the induction hypothesis lines \eqref{d 3} and \eqref{d 4} become
\begin{align}
\te(n,1,1)\te(n,0,0)^2+\te(n,0,1)\te(n,0,0)S_{-1}(\te(n,0,0)) \label{d2 1}\\ 
 -\te(n,1,1) h_1 \te(n,0,0)S_{-1}(\te(n,0,0)) \label{d2 2} \\ 
 + S_{-2}(\te(n,0,0))\te(n,0,0)^2 -S_{-2}(\te(n,0,0))h_1 \te(n,0,0)S_{-1}(\te(n,0,0)) \label{d2 3}.
\end{align}
Lines \eqref{d2 1} and \eqref{d2 2} are equal to \eqref{r1 3}, and added to \eqref{d 1} and \eqref{d 2} are equal to $\te(n,1,1)$. 

Applying the induction hypothesis to \eqref{d 7} yields 
\begin{equation}
 S_{-1}(\te(n,0,0))  S_{-2}(\te(n,0,0))  S_{-1}(\te(n,0,0)).  \label{d3 1} 
\end{equation}

Adding \eqref{d 5}, \eqref{d 6}, \eqref{d2 3}, and \eqref{d3 1} yields 
\begin{align*}
W_1(S_{-1}(\te(n,0,0)),\te(n,0,0),\te(n,0,0))  &= S_{-1}(W_1(\te(n,0,0),\te(n,0,0),\te(n,0,0)))\\ 
&= S_{-2}(W_0(\te(n,0,0),\te(n,0,0),\te(n,0,0)))\\ 
&= S_{-2}(\te(n+1,0,0)).
\end{align*}
This completes the induction step and the proof. 
\end{proof}

\begin{theorem} \label{t 1 main}
Let $M_{n,0}$ denote the multiset from Theorem \ref{t 0 main}. 
There exists multisets $M_{n,1}$ such that 
\[
\te(n,0,1) = \ti(M_{n,1})
\]
with $M_{n,1} \in R(2^{n+1}-1)$.
\begin{align}
\te(n+1,0,1) &= L(\ti(M_{n,0})) \ti(M_{n,0}+ S_{-1}(M_{n,0})) \nonumber  \\ 
&\,\,\,\,\,+ 2L(\ti(M_{n,0})) \ti(M_{n,0}+ M_{n,1})\nonumber \\ 
&\,\,\,\,\,+L(\ti(M_{n,1})) \ti(M_{n,0}+ M_{n,0}) \label{1 main form}.
\end{align}

\end{theorem} 
\begin{proof}
We use induction on $n$. We have 
\[
M_{0,1} = \emptyset.
\]
Assume the statements are true for some $n \geq 0$. 

In the definition $\te(n+1,0,1)$ of Definition \ref{d lc1 rec}, we apply Theorem \ref{t lc1 s}, equation \eqref{lc1 -1} and use the induction hypothesis to obtain 
 \begin{align*}
 \te(n+1,0,1) &= W_0(S_{-1}(\te(n,0,0)), \te(n,0,0), \te(n,0,0))\\
  &\,\,\,\,\,+ 2W_0(\te(n,0,1),\te(n,0,0),\te(n,0,1))\\ 
  &\,\,\,\,\,+W_0(\te(n,0,0), \te(n,0,1), \te(n,0,1)).
 \end{align*}
  From the proof of Theorem \ref{t W}, we use
 \[
 W_0(\ti(M_1), \ti(M_2), \ti(M_3)) = L(\ti(M_2))\ti(M_1+M_3).
 \]
 to obtain equation \eqref{1 main form}. 
 This implies that $M_{n+1,1}$ exists. Since $M_{n,0} \in R(2^{n+1})$ by Theorem \ref{t 0 main} and $M_{n,1} \in R(2^{n+1}-1)$ by the induction hypothesis, we have by Lemma \ref{l Rc add} 
 \begin{align*}
 M_{n,0}+ S_{-1}(M_{n,0}) &\in R(2^{n+2}-1)\\ 
 M_{n,0}+ M_{n,1} &\in R(2^{n+2}-1) \\ 
 M_{n,0}+ M_{n,0}  &\in R(2^{n+2}). 
 \end{align*}

By Lemma, \ref{l left Rc}, for any $c$, multiplying by $\ti_i$ for $i \geq 0$ preserves $R(c)$, and by Lemma \ref{l Rc add} 
\[
R(2^{n+2}) \subset R(2^{n+2}-1).
\]
Since $R(c)$ for any $c$ is closed under unions, we have that \[
M_{n+1,1}  = ( M_{n,0}+ S_{-1}(M_{n,0})) \cup (\bigcup_{i=1}^2 ( M_{n,0}+ M_{n,1}))\cup ( M_{n,0}+ M_{n,0}) \in R(2^{n+2}-1).
\] This completes the proof. 
\end{proof}
\begin{corollary} 

For $i \in \{ -1,0,1\}$
\[
L(\te(n,i,1)) \in \mathrm{CH}_2^+.
\]
\end{corollary} 

\begin{proof}
By Lemma \ref{l pos}, it is sufficient to show that these elements are in $\tCH(c)$ for some $c \geq 0$. $\te(n,0,1) \in \tCH(c)$ with $c = 2^{n+1}-1 \geq 0$. $\te(n,1,1) = S_{-1}(\te(n,0,1)) \in \tCH(c)$ with $c = 2^{n+1}-2\geq0$. And since $\te(n,-1,1) = \te(n,1,1)+ S_{-2}(\te(n,0,0))$ with $S_{-2}(\te(n,0,0)) \in \tCH(c)$ with $c =2^{n+1}-2 \geq 0$. This completes the proof. 
\end{proof}

\end{document}